\documentclass[a4paper,11pt]{article}
\usepackage[latin1]{inputenc}
\usepackage{amssymb}
\usepackage{amsmath}
\usepackage[english]{babel}
\newtheorem{theorem}{Theorem}[section]
\newtheorem{lemma}{Lemma}

\newtheorem{proposition}{Proposition}

\newenvironment{class}[1][2000 Mathematics Subject Classification]{\textbf{#1.} }{}
\newenvironment{MC}[1][Key words and phrases]{\textbf{#1.} }{}

\newenvironment{proof}[1][Proof]{\textbf{#1.} }{\ \rule{0.5em}{0.5em}}

\title{Martingale property for the Scott correlated stochastic volatility model}%
\author{\textbf{Khadija AKDIM }
\protect\footnote{k.akdim@uca.ma }
\textbf{M'hamed Eddahbi }
\protect\footnote{m.eddahhbi@uca.ma }
\textbf{Mouna Haddadi}
\protect\footnote{mouna.haddadi1@gmail.com }\\ 
Department of Mathematics,\\ 
Faculty of Sciences and Techniques,\\  Cadi Ayyad University, B.P. 549, 40.000 Marrakech, Morocco.\\ }
\date{}

\begin{document}
\maketitle
\begin{abstract}
In this paper, we study the martingale property for a Scott correlated stochastic volatility model, when the correlation coefficient between the Brownian motion driving the volatility and the one driving the
asset price process is arbitrary. For this study we verify the martingale property by using the necessary and sufficient conditions given by Bernard \emph{et al.} \cite{Bernard}.  Our main results are to prove that the price process is a true and uniformly integrable martingale if and only if $\rho \in [-1,0]$ for two transformations of Brownian motion describing the dynamics of the underling asset.
\end{abstract}
\begin{MC} 
Scott model, stochastic volatility, martingale property, local martingale.\end{MC}

\begin{class}60G44; 60H30\end{class}
\section{Introduction}
The very popular model for option pricing, it was established by Black and Scholes (1973) \cite{BS} (BS model hereafter). In particular, the BS model assumes that the underlying asset price follows a geometric Brownian
motion with a fixed volatility. Within the BS theory, the most direct technique constructs an equivalent martingale measure for the underlying asset process. However, the assumption of constant volatility was suspect from the beginning. Some statistical tests strongly reject the idea that a volatility process can be a constant. It also became clear, although this was less immediate, that the BS model was in conflict with evolving patterns in observed option pricing data. In particular, after the 1987 market crash, a persistent pattern emerged, called the \textquotedblleft smile\textquotedblright that should not exist under the BS theory. Nevertheless, the continuous-time framework provides several alternative models specially designed to explain, at least qualitatively, this effect. Among them, we highlight on the Stochastic Volatility (SV) models. These are two-dimensional diffusion processes in which one dimension describes the asset price dynamics and the second one governs the volatility evolution. The examples of stochastic volatility models are abundant: Hull and White \cite{HW}, Stein and Stein \cite{SS}, Heston \cite{HESTON}, Scott \cite{Scott}, Wiggins \cite{Wiggins}, Melino and Turnbull \cite{MT}. \\
The martingale problem has been extensively studied from Girsanov (1960)
who poses the problem of deciding whether a stochastic exponential is a true martingale or not.
In the context of stochastic volatility models, Bernard \emph{et al.} \cite{Bernard} have established necessary and sufficient analytic conditions to verify when a stochastic exponential of a continuous local martingale is a martingale or a uniformly integrable martingale for arbitrary correlation ($- 1\leq \rho \leq 1$).  Mijatovic and Urusov \cite{MU} have obtained necessary and sufficient conditions in the case of perfect correlation ($\rho = 1$), and Lions and Musiela \cite{LM} gave sufficient conditions to verify when a stochastic exponential of a continuous local martingale is a martingale or a uniformly integrable martingale, and also Sin \cite{Sin}, Andersen and Piterbarg \cite{and}, Bayraktar, Kardaras and Xing \cite{bay} provide easily verifiable conditions.\\
The Scott model assumes that the volatility process is the exponential of an Ornstein--Uhlenbeck stochastic process. \\ The Ornstein--Uhlenbeck model is able: (i) to describe simultaneously the observed long-range memory in volatility and the short one in leverage \cite{QZ}, (ii) to provide a consistent stationary distribution for the volatility with data \cite{CFMN, EPM}, (iii) it shows the same mean first--passage time profiles for the volatility as those of empirical daily data \cite{MP3} and finally (iv) it fairly reproduces the realized volatility having some degree of predictability in future return changes \cite{EPM}.\\
Our aim in the present work is to take advantage of all this knowledge to study the martingale property for the
Scott correlated stochastic volatility model. We shall use the criterium given by Bernard \emph{et al.}
\cite{Bernard} in two situations. The first one we use the Cholesky decomposition of the Brownian motion of the stock price as a linear transformation of two independent
Brownian motions. The second one consists to use transformations of Wu and Yor \cite{yor}. \\
The paper is organized as follows, in section 2, we recall some preliminary results and the main result of \cite{Bernard}. The section 3 is devoted to the study of the martingale property of the Scott model.

\section{Preliminaries}
We now formally introduce the setup of this work. We start by the presentation of general stochastic volatility model, and we introduce a canonical probability space of our processes, which we shall use to formulate the necessary and sufficient analytic conditions given by Bernard \emph{et al.} \cite{Bernard} to verify when a stochastic exponential of a continuous local martingale is a martingale or a uniformly integrable martingale for arbitrary correlation ($- 1\leq \rho \leq 1$).\\
We consider the state space $J=(\ell,r)$, $-\infty \leq \ell < r \leq \infty $, let the stochastic exponential $Z=(Z_{t})_{t\in \lbrack 0,\infty )}$ denote the (discounted) stock price  and a $J$--valued diffusion $Y=(Y_{t})_{t\in \lbrack 0,\infty \lbrack }$ on some filtered probability space $(\Omega,\mathcal{F},(\mathcal{F}_{t})_{t\in \lbrack 0,\infty \lbrack })$ governed by the stochastic differential equations for all $t \in [0,\zeta)$:
\begin{eqnarray}
\left\{
\begin{array}{l}
dZ_{t} =Z_{t}b(Y_{t})dW_{t}^{(1)}\text{, \ }Z_{0}=1 \label{sys1}  \\
dY_{t} =\mu (Y_{t})dt+\sigma (Y_{t})dW_{t},\text{ \, }Y_{0}=x_{0}\in \mathbb{R} \label{Y} \\
\end{array}%
\right.
\end{eqnarray}
where $W_{t}^{(1)}$ and $W_{t}$ are standard $\mathcal{F}_{t}$--Brownian motions, with $E [dW_{t}^{(1)} dW_{t}]= \rho dt$, $\rho$ is the constant correlation coefficient with $-1\leq \rho \leq 1$, denote $\zeta $ the exit time of $Y$  from its state space, where $ \zeta =\inf \{t>0:Y_{t}\notin J\}$, which mean that on:
\begin{itemize}
\item the event $\left\{ \zeta =\infty \right\} $ the trajectories of $Y$
do not exit $J$ ;
\item  the event $\{\zeta <\infty \}$, $\lim_{t\rightarrow \zeta
}Y_{t}=r$ or  $\lim_{t\rightarrow \zeta }Y_{t}=\ell$, $P$--a.s. $Y$  is
defined such that it stays at its exit point, which means that $\ell$ and $r$
are absorbing boundaries.
\end{itemize}
\textbf{Assumption H}:
Let $\mu $, $\sigma$ and $b : J\rightarrow \mathbb{R}$ be given Borel functions. Let $L_{loc}^{1}(J)$ denotes the class of locally integrable functions on $J$. We say that  $\mu $ and $\sigma$ satisfy: 
\begin{itemize}
\item[$\mathbf{(A1)}$]  if for all $x\in J$  $\sigma (x)\neq 0\text{ \ and  }\frac{1}{\sigma^{2}(\cdot)}\text{, \ }\frac{\mu (\cdot)}{\sigma ^{2}(\cdot)} \in L_{loc}^{1}(J)$.
\end{itemize}
And $b$ and $\sigma$ satisfy:
\begin{itemize}
\item[$\mathbf{(A2)}$] if $\frac{b^{2}(\cdot)}{\sigma ^{2}(\cdot)} \in L_{loc}^{1}(J)$.
\end{itemize}

Under condition $\mathbf{(A1)}$ the SDE satisfies by $Y$ defined in (\ref{Y}) has a unique solution in law  that possibly exits its state space $J$, and the condition $\mathbf{(A2)}$ ensures that the stochastic integral $ \int_{0}^{t\wedge \zeta }b(Y_{s})dW_{s}^{(1)}$ is well--defined, then the process $Z$ defined in (\ref{sys1}) is a nonnegative continuous local martingale.\\
We define the space accommodating all four processes ($Y$, $Z$, $W$, $W^{(1)}$).
\begin{itemize}
\item Let $\Omega _{1}:=\overline{\mathcal{C}}([0,\infty ),\overline{J})$  be the
space of continuous functions \\ $\omega _{1}$ $:[0,\infty )$ $\rightarrow
\overline{J}$  that start inside $J$ \ and can exit, i.e. there exists  $%
\zeta (\omega _{1})\in [0,\infty ]$  such that $\omega _{1}(t)\in J$ \ for $%
t<$ $\zeta (\omega _{1})$ and in the case $\zeta (\omega _{1})<\infty $ we have either $\omega _{1}(t)=r$ \ for $t\geq \zeta (\omega _{1})$ (hence also $\lim_{t\rightarrow \zeta (\omega
_{1})}\omega _{1}(t)=r$) or $\omega _{1}(t)=\ell$ \ for $t\geq $ $\zeta
(\omega _{1})$ (hence also $\lim_{t\rightarrow \zeta (\omega _{1})}\omega
_{1}(t)=\ell$).

\item Let $\Omega _{2}$ $:=\overline{\mathcal{C}}((0,\infty ),[0,\infty ])$ be the
space of continuous functions \\ $\omega _{2}$ $:(0,\infty ) \rightarrow
\lbrack 0,\infty ]$ with $\omega _{2}(0)=1$ that satisfy $\omega _{2}(t)=$ $%
\omega _{2}(t\wedge T_{0}(\omega _{2})\wedge T_{\infty }(\omega _{2}))$ for
all $t\geq 0$, where $T_{0}(\omega _{2})$ and $T_{\infty }(\omega _{2})$
denote the first hitting times of $0$ and $\infty $ by $\omega _{2}$.

\item Let $\Omega _{3}$ $=\overline{\mathcal{C}}([0,\infty ),(-\infty,\infty )
)$ be the space of continuous functions \\ $\omega _{3}:[0,\infty )\rightarrow
(-\infty,\infty ) $ \ with $\omega _{3}(0)=0$.

\item Let $\Omega _{4}$ $=\overline{\mathcal{C}}([0,\infty ),(-\infty,\infty )
)$ be the space of continuous functions \\ $\omega _{4}:[0,\infty )\rightarrow
(-\infty,\infty ) $ \ with $\omega _{4}(0)=0$.
\end{itemize}
Define the canonical process
\begin{center}
$(Y_{t}(\omega _{1}),Z_{t}(\omega
_{2}),W_{t}(\omega _{3}),W_{t}^{(1)}(\omega _{4})):=(\omega _{1}(t),\omega
_{2}(t),\omega _{3}(t),\omega _{4}(t))$
\end{center}
for all $t \geq 0$, and let $(\mathcal{F}_{t})_{t\geq 0}$ denote the filtration generated by the canonical process and satisfying the usual conditions, and $\sigma$--field is $\mathcal{F}=\bigvee _{t\in [0, \infty)}\mathcal{F}_{t}$. \\
Now, the processes are defined in this filtered space $(\Omega,\mathcal{F},(\mathcal{F}_{t})_{t\geq 0})$, let $\mathbb{P}$  be the probability
measure induced by the canonical process on the space $(\Omega,\mathcal{F})$.
\begin{proposition}(Change of measure for continuous local martingales)\label{P1}
Consider the space $\left( \Omega,\mathcal{F},(\mathcal{F})_{t\geq 0}\right) $, with the process $Z$ defined in (\ref{sys1}) and suppose that
the \textbf{Assumption H} is fulfilled. Then
\begin{itemize}
\item[1.] There exists a unique probability measure $\mathbb{Q}$ on
the same space such that, for any bounded stopping time $\tau $ and for all
non-negative $\mathcal{F} _{\tau }$--measurable random variables $S$,
\begin{eqnarray}\label{eq7}
E_{\mathbb{Q}}\left[ \frac{1}{Z_{\tau }}S\mathbf{1}_{\left\lbrace  0<Z_{\tau }<\infty \right\rbrace }%
\right] =E_{\mathbb{P}}\left[ S\mathbf{1}_{\left\lbrace 0<Z_{\tau}\right\rbrace }\right]
\end{eqnarray}
where we define $\frac{1}{Z_{\tau }}\mathbf{1}_{\left\lbrace 0<Z_{\tau }<\infty \right\rbrace }=0$ on $\{Z_{\tau
}=0\}$ from the usual convention.
\item[2.] Under $\mathbb{P}$, for $t\in \lbrack 0,T_{0}),$ define the continuous $\mathbb{P}$--local
martingale $M_{t}$ as:
\begin{eqnarray}
M_{t}=\int_{0}^{t\wedge \zeta }b(Y_{s})dW_{s}^{(1)}.
\end{eqnarray}

Then under  $\mathbb{Q}$  \, for $%
t\in \lbrack 0,T_{\infty })$,
\begin{eqnarray}
\widetilde{M}_{t}^{\ast }:=M_{t}-\left< M\right>_{t}
=\int_{0}^{t\wedge \zeta }b(Y_{s})dW_{s}^{(1)}-\int_{0}^{t\wedge \zeta
}b^{2}(Y_{s})ds
\end{eqnarray}
is a continuous   $\mathbb{Q}$--local martingale. Here $T_{0}$
and $T_{\infty }$  are defined as the first hitting times to $0$ and $\infty$ by $Z$.
\item[3.] Under  $\mathbb{Q}$, for $t\in \lbrack 0,T_{\infty })$%
\begin{eqnarray}
\frac{1}{Z_{t}} =\mathcal{E} (-\widetilde{M}_{t}^{\ast })=\exp \left\{
-\int_{0}^{t\wedge \zeta }b(Y_{s})dW_{s}^{(1)}+\frac{1}{2}\int_{0}^{t\wedge
\zeta }b^{2}(Y_{s})ds\right\}
\end{eqnarray}
\end{itemize}
\end{proposition}
\begin{proof}
The proof can be found in Ruf \cite{R} Theorem 2 and its proof.
\end{proof}
Fix an arbitrary constant $c\in J$  and introduce the scale functions $s(\cdot)$ of the SDE satisfies by $Y$ under $\mathbb{P}$, and $\widetilde{s}(\cdot)$ of the SDE satisfies by Y under $\mathbb{Q}$:
\[
s(x):=\int_{c}^{x}\exp \left\{ -\int_{c}^{y}\frac{2\mu }{\sigma ^{2}}%
(u)du\right\} dy\text{, \ }x\in \overline{J}
\]%

\[
\widetilde{s}(x):=\int_{c}^{x}\exp \left\{ -\int_{c}^{y}\frac{2 \widetilde{\mu} }{\sigma ^{2}}%
(u)du\right\} dy\text{, \ }x\in \overline{J}
\]
And introduce the following test functions for $x \in \overline{J}$, with a constant $c \in J$.
\begin{eqnarray*}
\upsilon (x)= 2 \int_{c}^{x} \frac{\left( s(x)-s(y)\right)}{s^{\prime }(y)\sigma ^{2}(y)}dy, \; \;
\upsilon _{b}(x)= 2 \int_{c}^{x} \frac{ \left( s(x)-s(y)\right) b^{2}(y)}{s^{\prime }(y)\sigma ^{2}(y)}dy
\end{eqnarray*}
\begin{eqnarray*}
\widetilde{\upsilon } (x)= 2 \int_{c}^{x} \frac{\left( \widetilde{s}(x)-\widetilde{s}(y)\right)}{\widetilde{s}^{\prime
}(y)\sigma ^{2}(y)}dy, \; \; \widetilde{\upsilon _{b}}(x)= 2 \int_{c}^{x} \frac{ \left( \widetilde{s}(x)-\widetilde{s}(y)\right) b^{2}(y)}{\widetilde{s}^{\prime }(y)\sigma ^{2}(y)}dy
\end{eqnarray*}
Consider the stochastic exponential $Z$ defined in (\ref{sys1}). The following proposition provides the necessary and sufficient condition for $Z_{T}$ to be a $\mathbb{P}$--martingale for all $T \in [0, \infty)$, when $-1\leq \rho \leq 1$. The proofs of the following propositions can be found in \cite{Bernard} (Propositions 4.1, 4.2, 4.3 and 4.4, p. 18--19)
\begin{proposition}\label{P2}
If \textbf{Assumption H} is satisfied, then for all $T \in [0,\infty)$, $\mathbb{E}^{\mathbb{P}}(Z_{T})=1$ if and only if at least one of the conditions $(A)$--$(D)$ below is satisfied:
\begin{itemize}
\item[(A)] $\widetilde{\upsilon} (\ell)$ $=\widetilde{\upsilon} (r)$ $=\infty $,

\item[(B)] $\widetilde{\upsilon} _{b}(r)$ $<\infty $ and $\widetilde{\upsilon} (r)$ $=\infty $,

\item[(C)] $\widetilde{\upsilon} _{b}(\ell)$ $<\infty $ and $\widetilde{\upsilon} (r)=\infty $,

\item[(D)] $\widetilde{\upsilon} _{b}(r)$ $<\infty $ and $\widetilde{\upsilon} _{b}(\ell)$ $<\infty $.
\end{itemize}
\end{proposition}
We have the following necessary and sufficient condition for $Z$ to be a uniformly integrable
$\mathbb{P}$--martingale on $[0, \infty)$, when $ -1 \leq \rho \leq 1$.
\begin{proposition}\label{P3}
If \textbf{Assumption H} is satisfied, then $\mathbb{E}^{\mathbb{P}}( Z_{\infty })= 1$  if and only if at
least one of the conditions $(A')$--$(D')$ below is satisfied:
\begin{itemize}
\item[(A')] $b = 0$ a.e. on $J$ with respect to the Lebesgue measure,

\item[(B')] $\widetilde{\upsilon} _{b}(r)<\infty $ and $\widetilde{s}(\ell)=-\infty $,

\item[(C')] $\widetilde{\upsilon} _{b}(\ell)<\infty $ and $\widetilde{s}(r)=\infty $,

\item[(D')] $\widetilde{\upsilon} _{b}(r)<\infty $ and $\widetilde{\upsilon} _{b}(\ell)<\infty $.
\end{itemize}
\end{proposition}
\begin{proposition}\label{P4}
If \textbf{Assumption H} is satisfied, then for all $T \in [0,\infty)$, $Z_{T} > 0$ $\mathbb{P}$--a.s. if and
only if at least one of the conditions 1.--4. below is satisfied:
\begin{itemize}
\item[1.] $\upsilon (\ell)$ $=\upsilon (r)$ $=\infty $,

\item[2.] $\upsilon _{b}(r)$ $<\infty $ and $\upsilon (r)$ $=\infty $,

\item[3.] $\upsilon _{b}(\ell)$ $<\infty $ and $\upsilon (r)=\infty $,

\item[4.] $\upsilon _{b}(r)$ $<\infty $ and $\upsilon _{b}(\ell)$ $<\infty $.
\end{itemize}
\end{proposition}
\begin{proposition}\label{P5}
If \textbf{Assumption H} is satisfied, and let $Y$ be a (possibly explosive) solution of the SDE (\ref{Y}) under $\mathbb{P}$, with $Z$ defined in (\ref{sys1}), then $Z_{\infty} > 0$, $\mathbb{P}$--a.s. if and only if at least one of the conditions 1.--4. below is satisfied:
\begin{itemize}
\item[1.] $b = 0$ a.e. on $J$ with respect to the Lebesgue measure,

\item[2.] $\upsilon _{b}(r)<\infty $ and $s(\ell)=-\infty $,

\item[3.] $\upsilon _{b}(\ell)<\infty $ and $s(r)=\infty $,

\item[4.] $\upsilon _{b}(r)<\infty $ and $\upsilon _{b}(\ell)<\infty $.
\end{itemize}
\end{proposition}
\section{Main results}
In this section, we apply the results of Bernard \emph{et al.}  \cite{Bernard} to the study of martingale properties
of (discounted) stock prices in Scott correlated stochastic volatility model \cite{Scott} in two cases, the first one
by using the Cholesky decomposition, and the second one by using a transformation given by
Wu and Yor \cite{yor}.
\subsection{Cholesky decomposition}
Let $(\Omega,\mathcal{F},\mathbb{P})$ be a complete probability space and let $(W_t)_{t \geq 0}$ be a standard Brownian motion with respect to the filtration ($\mathcal{F})_{t \geq 0}$.
Let $(B_t)_{t \geq 0}$ be another standard Brownian motion on the same probability space
which is independent of $(W_t)_{ t \geq 0 }$.
\begin{proposition}{(The Cholesky decomposition)}
The linear transformation $T^{\rho}$ for $ \rho \in [-1,1]$, defined by
\begin{eqnarray*}
T^{\rho}_t = \rho W_t -\sqrt{1 - \rho ^2} B_t,
\end{eqnarray*}
defines a new Brownian motion $(\Omega,\mathcal{F},\mathbb{P})$.
\end{proposition}
On a filtered probability space $\left( \Omega,\mathcal{F},(\mathcal{F})_{t\in \left[0,\infty \right) }\right) $, we consider the following risk--neutral Scott model for the actualized asset price $S_{t}$:
\begin{eqnarray}\label{SYSC}
\left\{
\begin{array}{l}
dS_{t}=\sigma _{t}S_{t} \mathbf{1}_{[0,\zeta)} (t) dT_{t}^{\rho} \\
\sigma_{t}=f(Y_{t})=e^{Y_{t}} \\
dY_{t} = \alpha (m -Y_{t})\mathbf{1}_{[0,\zeta)} (t) dt+\beta \mathbf{1}_{[0,\zeta)} (t) dW_{t}%
\end{array}%
\right.
\end{eqnarray}
where $E \left[dT_{t}^{\rho} dW_{t} \right] = \rho dt$, and $-1\leq \rho \leq 1$, $\alpha > 0$, $ m > 0$, $\beta > 0$.
The natural state space for $Y$ is $J = (\ell,r) = (0,+\infty)$. $\zeta$ is the possible exit time of the process $Y$ from $J$.
The Scott model belongs to the general stochastic volatility model considered in (\ref{sys1}) and (\ref{Y}) with $\mu (x) =\alpha (m-x)$, $\sigma (x) = \beta$ and $b(x) = e^{x} $. \\
Since
\begin{eqnarray*}
&&\forall \;  x\in J, \; \; \; \sigma (x) \neq 0, \; \; \; \frac{1}{%
\sigma ^{2}(x)}=\frac{1}{\beta ^{2}}\in L_{loc}^{1}(J), \\
&&\frac{\mu (x)}{\sigma ^{2}(x)} =\frac{\alpha (m-x)}{\beta ^{2}}%
\in L_{loc}^{1}(J), \; \frac{b^{2}(x)}{\sigma ^{2}(x)}=\frac{%
e^{2x}}{\beta ^{2}}\in L_{loc}^{1}(J).
\end{eqnarray*}
Then the \textbf{Assumption H} is satisfied. From Proposition \ref{P1}, there exists a probability $\mathbb{Q}$ is absolutely continuous with respect to $\mathbb{P}$.
\begin{lemma}\label{L7}
If \textbf{Assumption H} is satisfied, then $\zeta \leq T_{\infty}$, $\mathbb{P}$--a.s. and $\mathbb{Q}$--a.s, where $ T_{\infty}$ is the first hitting times of $\infty$ by $Z$.
\end{lemma}
\begin{proof}
The proof can be found in Lemma 2.4 p. 9 \cite{Bernard}.
\end{proof}
\begin{proposition}
Under $\mathbb{Q}$, if \textbf{Assumption H} is satisfied, the diffusion $Y$ satisfies the following SDE up to $\zeta $
\begin{eqnarray}\label{YQC}
dY_{t}&=&\left( \mu (Y_{t})+ \rho b(Y_{t})\sigma (Y_{t})\right) \mathbf{1}_{t\in \lbrack 0,\zeta
]}dt+\sigma (Y_{t})\mathbf{1}_{t\in \lbrack 0,\zeta ]}d\widetilde{W_{t}},\\
 Y_{0} &=&x_{0}\nonumber
\end{eqnarray}
where $\widetilde{W}$ is a standard $\mathbb{Q}$--Brownian motion.
\end{proposition}
\begin{proof}
Denote $R_n$ as the first hitting time of $S$ to the level $n$ and set $\tau _n = R_n \wedge n$ for all $n \in \mathbb{N}$. Define $\zeta _n = \zeta \wedge \tau _n$, and consider the process $\widetilde{W}$ up to $\zeta _n $. Since $ \mathcal{F}_{\zeta _n } \subset \mathcal{F}_{\tau _n }$, it follows from Proposition \ref{P1} that $\mathbb{Q}$ restricted to $\mathcal{F}_{\zeta _n }$ is absolutely continuous with respect to $P$ restricted to
$\mathcal{F}_{\zeta _n }$ for $n \in \mathbb{N}$. Then from Girsanov Theorem
\begin{eqnarray*}
\widetilde{W}_t &:=& W_t - \left< W_\cdot,  \int_{0}^{\cdot} b(Y_{s})dT^{\rho}_s  \right>_t \\
&=& W_t - \left< W_{\cdot}, \rho \int_{0}^{\cdot} b(Y_{s}) dW_s  \right>_t +  \left< W_\cdot,  \sqrt{1- \rho ^2}\int_{0}^{\cdot} b(Y_{s})dB_{s}   \right>_t  \\
&=& W_t - \rho \int_{0}^{t} b(Y_{s}) ds
\end{eqnarray*}
is $\mathbb{Q}$--Brownian motion for $t \in [0, \zeta _n )$ and $n \in \mathbb{N}$.\\
From monotone convergence, $\mathbb{Q}(\lim _{n\rightarrow \infty} \tau _n = T_{\infty})$ and $\mathbb{Q}(\lim _{n\rightarrow \infty} \zeta _n = \zeta \wedge T_{\infty})$ hold. From Lemma \ref{L7},
 $\mathbb{Q}(\lim _{n\rightarrow \infty} \zeta _n = \zeta) = 1$,
Thus $Y$ is governed by the following SDE under $\mathbb{Q}$ for $t \in [0, \zeta)$
\begin{eqnarray*}
dY_{t}&=& \mu( Y_{t})dt + \sigma (Y_{t}) \left( d\widetilde{W}_t +\rho b(Y_{t}) dt\right)  \\
&=&\left( \mu (Y_{t})+ \rho b(Y_{t})\sigma (Y_{t})\right)dt+\sigma (Y_{t})d\widetilde{W_{t}}
\end{eqnarray*}
\end{proof}
For a constant $c \in J$, we calculate the scale functions of the SDE (\ref{SYSC}) and SDE (\ref{YQC}) for $x\in J $:
\begin{eqnarray*}
s(x) &=&\int_{c}^{x}\exp \left( -\int_{c}^{y}\frac{2\mu }{\sigma ^{2}}%
(z)dz\right) dy\\
&=&\int_{c}^{x}\exp \left( -\int_{c}^{y}\frac{2\alpha (m-z)}{\beta ^{2}}%
dz\right) dy \\
&=&\exp \left( -\frac{\alpha }{\beta ^{2}}(c-m)^{2}\right) \int_{c}^{x}\exp
\left( \frac{\alpha }{\beta ^{2}}(y-m)^{2}\right) dy \\
&=&A_{1}\int_{c}^{x}\exp\left(\frac{\alpha }{\beta ^{2}}(y-m)^{2}\right)dy,
\end{eqnarray*}
and \begin{eqnarray*}
\widetilde{s}(x) &=&\int_{c}^{x}\exp \left( -\int_{c}^{y}\frac{2\widetilde{%
\mu }}{\widetilde{\sigma }^{2}}(z)dz\right) dy\text{ \ \ ; \ }x\in J \\
&=&\int_{c}^{x}\exp \left( -\int_{c}^{y}\frac{2\alpha (m-z+\frac{\rho \beta
}{\alpha }e^{z})}{\beta ^{2}}dz\right) dy \\
&=&\exp \left( -\frac{\alpha }{\beta ^{2}}(c-m)^{2}+\frac{2\rho }{\beta }%
e^{c}\right) \int_{c}^{x}\exp \left( \frac{\alpha }{\beta ^{2}}%
(y-m)^{2}\right) \exp \left( -\frac{2\rho }{\beta }e^{y}\right) dy \\
&=&A_{2}\int_{c}^{x}e^{\frac{\alpha }{\beta ^{2}}(y-m)^{2}}e^{-\frac{2\rho }{%
\beta }e^{y}}dy,
\end{eqnarray*}
where $A_{1}=\exp( -\frac{\alpha }{\beta ^{2}}(c-m)^{2})$ and $A_{2}=\exp( -\frac{\alpha }{\beta ^{2}}(c-m)^{2}+\frac{2\rho }{\beta }e^{c})$.\\
Under $\mathbb{P}$ and $\mathbb{Q}$, we calculate the test functions for $x \in J$:
\begin{eqnarray*}
\upsilon (x) =\int_{c}^{x}\left( s(x)-s(y)\right) \frac{2}{s^{\prime
}(y)\sigma ^{2}(y)}dy %
=\frac{2}{\beta ^{2}}\int_{c}^{x}\frac{\left( \int_{y}^{x}e^{\frac{\alpha
}{\beta ^{2}}(z-m)^{2}}dz\right) }{e^{\frac{\alpha }{\beta ^{2}}(y-m)^{2}}}dy
\end{eqnarray*}
\begin{eqnarray*}
\upsilon_{b} (x) =\int_{c}^{x}\left( s(x)-s(y)\right) \frac{2b^{2}(y)}{%
s^{\prime }(y)\sigma ^{2}(y)}dy %
=\frac{2}{\beta ^{2}}\int_{c}^{x}\frac{\left( \int_{y}^{x}e^{\frac{\alpha
}{\beta ^{2}}(z-m)^{2}}dz\right) e^{2y} }{e^{\frac{\alpha }{\beta ^{2}}(y-m)^{2}}}%
dy
\end{eqnarray*}
\begin{eqnarray*}
\widetilde{\upsilon }(x) =\int_{c}^{x}\left( \widetilde{s}(x)-\widetilde{s}%
(y)\right) \frac{2}{\widetilde{s}^{\prime }(y)\sigma ^{2}(y)}dy %
=\frac{2}{\beta ^{2}}\int_{c}^{x}\frac{\left( \int_{y}^{x}e^{\frac{\alpha
}{\beta ^{2}}(z-m)^{2}}e^{-\frac{2\rho }{\beta }e^{z}}dz\right) }{e^{\frac{%
\alpha }{\beta ^{2}}(y-m)^{2}}e^{-\frac{2\rho }{\beta }e^{y}}}dy
\end{eqnarray*}
\begin{eqnarray*}
\widetilde{\upsilon }_{b}(x) &=&\int_{c}^{x}\left( \widetilde{s}(x)-%
\widetilde{s}(y)\right) \frac{2b^{2}(y)}{\widetilde{s}^{\prime }(y)\sigma
^{2}(y)}dy \\
&=&\frac{2}{\beta ^{2}}\int_{c}^{x}\frac{\left( \int_{y}^{x}e^{\frac{\alpha
}{\beta ^{2}}(z-m)^{2}}e^{-\frac{2\rho }{\beta }e^{z}}dz\right)e^{2y} }{e^{\frac{%
\alpha }{\beta ^{2}}(y-m)^{2}}e^{-\frac{2\rho }{\beta }e^{y}}}dy
\end{eqnarray*}
By using the Cholesky decomposition, we have the followings results,
\begin{theorem}\label{T9}
For the Scott model (\ref{SYSC}), the underlying stock price $(S_{t})_{0\leq t\leq T}$; $T\in [0, \infty)$ is a true $\mathbb{P}$--martingale\footnote{The same result is proved by B. Jourdain \cite{j}} if and only if $\rho \leq 0$.
\end{theorem}
\begin{proof}
To prove this theorem, We will check that one of the conditions (A)--(D) of the Proposition \ref{P2} is satisfied.\\
The proof detail is given in Appendix \ref{Asubsec:1}. 
The results are summarized in the following table:
\begin{table}{Summary Table}{tab:1 }\label{t1}
\begin{tabular}{|c|c|c|c|c|c|c|}
\hline
Case & $\widetilde{s}(0)$ & $\widetilde{s}(\infty )$ & $\widetilde{\upsilon }(0)$ & $\widetilde{\upsilon }(\infty )$ & $\widetilde{\upsilon }_{b}(0)$ & $ \widetilde{\upsilon }_{b}(\infty )$ \\
\hline
$\rho \leq 0$ & $>-\infty $ & $+\infty $ & $<+\infty $ & $+\infty $ & $<+\infty $ & $+\infty $ \\
\hline
$\rho >0$ & $>-\infty $ & $<+\infty $ & $<+\infty $ & $<+\infty $ & $<+\infty $ & $+\infty $\\
\hline
\end{tabular}
\end{table}
Therefore the condition (C) of Proposition \ref{P2} is fulfilled, then we conclude that $(S_{t})_{0 \leq t\leq T}$ is a true $\mathbb{P}$--martingale if and only if $\rho \leq 0$.
\end{proof}
\begin{theorem}\label{T10}
For the Scott model (\ref{SYSC}), the underlying stock price $(S_{t})_{0\leq t\leq T}$; $T\in [0, \infty)$ is a uniformly integrable $\mathbb{P}$--martingale if and only if $\rho \leq 0$.
\end{theorem}
\begin{proof}
Similar to the proof of Theorem \ref{T9}, from the Table \ref{t1}, we have  $\widetilde{\upsilon} _{b}(0)<\infty $ and $\widetilde{s}(\infty)=\infty $ for all $\rho \leq 0$, which is the condition (C') of Proposition \ref{P3}, then we deduce that $(S_{t})_{0\leq t\leq T}$ is a uniformly integrable $\mathbb{P}$--martingale if and only if $\rho \leq 0$.
\end{proof}
\begin{theorem}
For the Scott model (\ref{SYSC}), we have for all $\rho \in [-1,1]$:
\begin{equation*}
\mathbb{P}(S_{T}>0)=1, \; \; \text{for all}\; T\in [0, \infty) .
\end{equation*}
\end{theorem} 
\begin{proof}
We prove that at least one of the conditions 1.--4. of the Proposition \ref{P4} is satisfied.\\
The proof detail is given in Appendix \ref{Asubsec:2}. 
The results are summarized in the following table,
\begin{table}{Summary Table}{tab: 2}\label{t2}
\begin{tabular}{|c|c|c|c|c|c|}
\hline
$s(0)$ & $s(\infty )$ & $\upsilon(0)$ & $\upsilon(\infty )$ & $\upsilon_{b}(0)$ & $ \upsilon_{b}(\infty )$ \\
\hline
 $>-\infty $ & $+\infty $ & $<+\infty $ & $+\infty $ & $<+\infty $ & $+\infty $ \\
\hline
\end{tabular}
\end{table}
Since the condition 3. of the Proposition \ref{P4} is satisfied, then $\mathbb{P}(S_{T}>0)=1$ for all $T\in [0, \infty)$
\end{proof}
\begin{theorem}
For the Scott model (\ref{SYSC}), we have for all $\rho \in [-1,1]$:
\begin{equation*}
\mathbb{P}(S_{\infty}>0)<1 .
\end{equation*}
\end{theorem}
\begin{proof}
From the Table \ref{t2}, we have $\upsilon _{b}(0)<\infty $ and $s(\infty)=\infty $.
Thus the condition 3. of Proposition \ref{P5} is satisfied, then $\mathbb{P}(S_{\infty}>0)<1$.
\end{proof}
\subsection{Transformations of Wu and Yor}
Now we shall use a linear transformations of two independent Brownian motions given by Wu and Yor \cite{yor}.
\begin{proposition}{$($Theorem 2.1 \cite{yor}$)$}
Let $(\Omega,\mathcal{F},\mathbb{P})$ be a complete probability space, let $(W_t)_{t \geq 0}$ and $(B_t)_{t \geq 0}$ be two independent Brownian motions with respect to ($\mathcal{F} )_{t \geq 0}$. We consider the transformation $T^{\rho}$ for $ \rho \in [0,1]$, defined by
\begin{eqnarray*}
 T^{\rho}_t = W_t - \int ^t_0 \left( \frac{1-\rho}{s} W_s + \frac{\sqrt{\rho - \rho ^2}}{s} B_s \right) ds
\end{eqnarray*}
then $T^{\rho}$ is a new Brownian motion.
\end{proposition}
With this new transformation $T^{\rho}$, we consider the following risk--neutral Scott model for the discounted assets price process $S_{t}$:
\begin{eqnarray}\label{SYSWY}
\left\{
\begin{array}{l}
dS_{t}=\sigma _{t}S_{t} \mathbf{1}_{[0,\zeta)}(t) dT_{t}^{\rho} \\ 
\sigma_{t}=f(Y_{t})=e^{Y_{t}} \\ 
dY_{t}=\alpha (m -Y_{t})\mathbf{1}_{[0,\zeta)}(t) dt+\beta \mathbf{1}_{[0,\zeta)}(t) dW_{t}%
\end{array}%
\right.
\end{eqnarray}
with $E \left[dT_{t}^{\rho} dW_{t} \right] = \rho dt$ and $ 0 \leq \rho \leq 1$, $\alpha > 0$, $ m > 0$, $\beta > 0$.
The natural state space for $Y$ is $J = (\ell,r) = (0,+\infty)$. $\zeta$ is the possible exit time of the process $Y$ from its state space $J$.
\begin{proposition}
If \textbf{Assumption H} is satisfied under $\mathbb{Q}$, then the diffusion $Y$  satisfies the
following SDE up to $\zeta $
\begin{eqnarray}\label{YQWY}
dY_{t}&=&\left( \mu (Y_{t})+ b(Y_{t})\sigma (Y_{t})\right) \mathbf{1}_{ \lbrack 0,\zeta
]}(t)dt+\sigma (Y_{t}) \mathbf{1}_{ \lbrack 0,\zeta ]}(t)d\widetilde{W_{t}}\\
Y_{0} &=&x_{0} \nonumber
\end{eqnarray}
where $\widetilde{W}$ is a standard $\mathbb{Q}$--Brownian motion.
\end{proposition}
\begin{proof}
Denote $R_n$ as the first hitting time of $S$ to the level $n$, and set $\tau _n = R_n \wedge n$ for all $n \in \mathbb{N}$.
Define $\zeta _n = \zeta \wedge \tau _n$, and consider the process $\widetilde{W}$ up to $\zeta _n $. Since $ \mathcal{F}_{\zeta _n } \subset \mathcal{F}_{\tau _n }$, it follows from Proposition \ref{P1} that $\mathbb{Q}$ restricted to $\mathcal{F}_{\zeta _n }$ is absolutely continuous with respect to $\mathbb{P}$ restricted to
$\mathcal{F}_{\zeta _n }$ for $n \in \mathbb{N}$. Then from Girsanov Theorem
\begin{eqnarray*}
\widetilde{W}_t &:=& W_t - \left< W_\cdot,  \int_{0}^{\cdot} b(Y_{s})dT^{\rho}_s  \right>_t \\
&=& W_t - \left< W_\cdot,  \int_{0}^{\cdot} b(Y_{s}) dW_s  \right>_t +  \left< W_\cdot,  \int_{0}^{\cdot} (1- \rho )%
\frac{b(Y_{s})}{s} W_{s} ds  \right>_t  \\
& &  + \left< W_\cdot,  \int_{0}^{\cdot} \sqrt{\rho - \rho ^2}  \frac{b(Y_{s})}{s} B_{s}  ds  \right>_t \\
&=& W_t - \int_{0}^{t} b(Y_{s}) ds
\end{eqnarray*}
is $\mathbb{Q}$--Brownian motion for $t \in [0, \zeta _n )$ and $n \in \mathbb{N}$.\\
We have $\mathbb{Q}(\lim _{n\rightarrow \infty} \zeta _n = \zeta) = 1$,
Thus $Y$ is governed by the following SDE under $\mathbb{Q}$ for $t \in [0, \zeta)$
\begin{eqnarray*}
dY_{t}&=& \mu( Y_{t})dt + \sigma (Y_{t}) \left( d\widetilde{W}_t + b(Y_{t}) dt\right)  \\
&=&\left( \mu (Y_{t})+ b(Y_{t})\sigma (Y_{t})\right)dt+\sigma (Y_{t})d\widetilde{W_{t}}
\end{eqnarray*}
\end{proof}
For a constant $c \in J$, we calculate the scale functions of the SDE (\ref{SYSWY}) and SDE (\ref{YQWY}), for any $x\in J$:
\begin{eqnarray*}
s(x) &=&\int_{c}^{x}\exp \left( -\int_{c}^{y}\frac{2\mu }{\sigma ^{2}}%
(z)dz\right) dy \\
&=&\int_{c}^{x}\exp \left( -\int_{c}^{y}\frac{2\alpha (m-z)}{\beta ^{2}}%
dz\right) dy \\
&=&A_{1}\int_{c}^{x}e^{\frac{\alpha }{\beta ^{2}}(y-m)^{2}}dy,
\end{eqnarray*}
and \begin{eqnarray*}
\widetilde{s}(x) &=&\int_{c}^{x}\exp \left( -\int_{c}^{y}\frac{2\widetilde{%
\mu }}{\widetilde{\sigma }^{2}}(z)dz\right) dy \\
&=&\int_{c}^{x}\exp \left( -\int_{c}^{y}\frac{2\alpha (m-z+\frac{ \beta
}{\alpha }e^{z})}{\beta ^{2}}dz\right) dy \\
&=&A_{2}\int_{c}^{x}e^{\frac{\alpha }{\beta ^{2}}(y-m)^{2}}e^{-\frac{2 }{%
\beta }e^{y}}dy,
\end{eqnarray*}
where $A_{1}=\exp( -\frac{\alpha }{\beta ^{2}}(c-m)^{2})$
and $ A_{2}=\exp( -\frac{\alpha }{\beta ^{2}}(c-m)^{2}+\frac{2 }{\beta }e^{c})$.\\
Under $\mathbb{P}$, we calculate the test functions for $x \in \bar{J}$:
\begin{eqnarray*}
\upsilon (x)& =& \int_{c}^{x}\left( s(x)-s(y)\right) \frac{2}{s^{\prime
}(y)\sigma ^{2}(y)}dy \\
&=& \frac{2}{\beta ^{2}}\int_{c}^{x}\frac{\left( \int_{y}^{x}e^{\frac{\alpha
}{\beta ^{2}}(z-m)^{2}}dz\right) }{e^{\frac{\alpha }{\beta ^{2}}(y-m)^{2}}}dy
\end{eqnarray*}
\begin{eqnarray*}
\upsilon_{b} (x) &=&\int_{c}^{x}\left( s(x)-s(y)\right) \frac{2b^{2}(y)}{%
s^{\prime }(y)\sigma ^{2}(y)}dy \\
&=&\frac{2}{\beta ^{2}}\int_{c}^{x}\frac{\left( \int_{y}^{x}e^{\frac{\alpha
}{\beta ^{2}}(z-m)^{2}}dz\right) e^{2y} }{e^{\frac{\alpha }{\beta ^{2}}(y-m)^{2}}}%
dy.
\end{eqnarray*}
Under $\mathbb{Q}$, we calculate the test functions for $x \in \bar{J}$:
\begin{eqnarray*}
\widetilde{\upsilon }(x) &=&\int_{c}^{x}\left( \widetilde{s}(x)-\widetilde{s}%
(y)\right) \frac{2}{\widetilde{s}^{\prime }(y)\sigma ^{2}(y)}dy \\
&=& \frac{2}{\beta ^{2}}\int_{c}^{x}\frac{\left( \int_{y}^{x}e^{\frac{\alpha
}{\beta ^{2}}(z-m)^{2}}e^{-\frac{2 }{\beta }e^{z}}dz\right) }{e^{\frac{%
\alpha }{\beta ^{2}}(y-m)^{2}}e^{-\frac{2 }{\beta }e^{y}}}dy
\end{eqnarray*}
\begin{eqnarray*}
\widetilde{\upsilon }_{b}(x) &=&\int_{c}^{x}\left( \widetilde{s}(x)-%
\widetilde{s}(y)\right) \frac{2b^{2}(y)}{\widetilde{s}^{\prime }(y)\sigma
^{2}(y)}dy \\
&=& \frac{2}{\beta ^{2}}\int_{c}^{x}\frac{\left( \int_{y}^{x}e^{\frac{\alpha
}{\beta ^{2}}(z-m)^{2}}e^{-\frac{2 }{\beta }e^{z}}dz\right)e^{2y} }{e^{\frac{%
\alpha }{\beta ^{2}}(y-m)^{2}}e^{-\frac{2 }{\beta }e^{y}}}dy
\end{eqnarray*}
By using the Transformations of Wu and Yor, we have the followings results,
\begin{theorem}\label{T15}
For the Scott model (\ref{SYSWY}), the underlying stock price $(S_{t})_{0\leq t\leq T}$; $T\in [0, \infty)$ is not a
true $\mathbb{P}$--martingale if and only if $\rho \geq 0$.
\end{theorem}
\begin{proof}
To prove this theorem, we show this contrapositive of the Proposition \ref{P2}: if $(\lbrace \widetilde{\upsilon }(0) < \infty \rbrace$ and $\lbrace \widetilde{\upsilon }_b (0) = \infty \rbrace)$ or $(\lbrace \widetilde{\upsilon }(\infty) < \infty \rbrace$ and $\lbrace \widetilde{\upsilon }_b (\infty) = \infty \rbrace)$, then
\begin{table}{Summary Table}{tab: 3}\label{t3}
\begin{tabular}{|c|c|c|c|c|c|}
\hline
 $\widetilde{s}(0)$ & $\widetilde{s}(\infty )$ & $\widetilde{\upsilon }(0)$ &  $\widetilde{\upsilon }(\infty )$ &  $\widetilde{\upsilon }_{b}(0)$ & $ \widetilde{\upsilon }_{b}(\infty )$ \\
\hline
$>-\infty $ & $<+\infty $ & $<+\infty $ & $<+\infty $ & $<+\infty $ & $+\infty $ \\
\hline
\end{tabular}
\end{table}
Thus from  this table, $(\lbrace \widetilde{\upsilon }(\infty) < \infty \rbrace$ and $\lbrace \widetilde{\upsilon }_b (\infty) = \infty \rbrace)$ is satisfied, then $(S_{t})_{0\leq t\leq T}$ is  not a true $\mathbb{P}$--martingale.
\end{proof}
\begin{theorem}
For the Scott model (\ref{SYSWY}), the underlying stock price $(S_{t})_{0\leq t\leq T}$; $T\in [0, \infty)$ is not uniformly integrable $\mathbb{P}$--martingale if and only if $\rho \geq 0$.
\end{theorem}
\begin{proof}
We check this contrapositive of  Proposition \ref{P3}: if $(\lbrace \widetilde{s}(0) > -\infty \rbrace$ and $\lbrace \widetilde{\upsilon }_b (0) = \infty \rbrace)$ or $(\lbrace \widetilde{s }(\infty) < \infty \rbrace$ and $\lbrace \widetilde{\upsilon }_b (\infty) = \infty \rbrace)$.\\
Thus from the Table \ref{t3} , we have $(\lbrace \widetilde{s }(\infty) < \infty \rbrace$ and $\lbrace \widetilde{\upsilon }_b (\infty) = \infty \rbrace)$, then $(S_{t})_{0\leq t\leq T}$ is not uniformly integrable $\mathbb{P}$--martingale.
\end{proof}
\begin{theorem}
For the Scott model (\ref{SYSWY}), $\mathbb{P}(S_{T}>0)=1$ for all $T\in [0, \infty)$
\end{theorem}
\begin{proof}
We check that at least one of the conditions 1.--4. of the Proposition \ref{P4} is satisfied.\\
We have
\begin{table}{Summary Table}{tab:4}\label{t4}
\begin{tabular}{|c|c|c|c|c|c|}
\hline
$s(0)$ & $s(\infty )$ & $\upsilon(0)$ & $\upsilon(\infty )$ & $\upsilon_{b}(0)$ & $ \upsilon_{b}(\infty )$ \\
\hline
$ >-\infty $ & $+\infty $ & $<+\infty $ & $+\infty $ & $<+\infty $ & $+\infty $ \\
\hline
\end{tabular}
\end{table}
Since the condition 3. of Proposition \ref{P4} is satisfied, then $\mathbb{P}(S_{T}>0)=1$ for all $T\in [0, \infty)$.
\end{proof}
\begin{theorem}
For the Scott model (\ref{SYSWY}), $\mathbb{P}(S_{\infty}>0)<1$.
\end{theorem}
\begin{proof}
We will show that one of the conditions 1.--4. of the Proposition \ref{P5} is satisfied.\\
From the Table \ref{t4}, we have $\upsilon _{b}(0)<\infty $ and $s(\infty)=\infty $.
Therefore the condition 3. of Proposition \ref{P5} is satisfied, then $\mathbb{P}(S_{\infty}>0)<1$.
\end{proof}

\section*{Conclusion}

In this paper we have proved by using two linear transformation of the two independent Brownian motion, which were known as the Cholesky decomposition and Wu and Yor transformation,  that the  stock price process is a true and a uniformly integrable martingale if and only if $\rho \in [-1,0]$ (see Theorem \ref{T9} and Theorem \ref{T15}).
Therefore in the Scott correlated stochastic volatility model, the stock price is a true martingale if and only
 if $\rho \in [-1,0]$.

\appendix{Proof of Theorem \ref{T9}}\label{Asubsec:1}

For the sake of simplification of the notations we set $f(x)=e^{\frac{\alpha }{\beta ^{2}}(x-m)^{2}-\frac{2\rho }{\beta}e^{x}}$.
Under the probability measure $\mathbb{Q}$, we have a scale function:
\begin{eqnarray*}
\widetilde{s}(x) =A_{2}\int_{c}^{x}f(y)dy, \quad x\in J
\end{eqnarray*}
Let us check the conditions for $r$, recall that $r=\infty$
\begin{itemize}
\item Case (1): $\rho \leq 0$ \\
By integration by parts, one has:
$\widetilde{s}(\infty )=A_{2}\int_{c}^{\infty }f(y)dy$, for all $x \in ]c,\infty[$:
\begin{eqnarray*}
\int_{c}^{x}f(y)dy &=&\int_{c}^{x}\frac{1}{\left( \frac{2\alpha }{\beta ^{2}}%
(y-m)-\frac{2\rho }{\beta }e^{y}\right) }\left( \frac{2\alpha }{\beta ^{2}}%
(y-m)-\frac{2\rho }{\beta }e^{y}\right) f(y)dy \\
&=&\left[ \frac{f(y)}{\left( \frac{2\alpha }{\beta ^{2}}(y-m)-\frac{2\rho }{%
\beta }e^{y}\right) }\right] _{c}^{x}+\int_{c}^{x}\frac{\frac{2\alpha }{%
\beta ^{2}}-\frac{2\rho }{\beta }e^{y}}{\left( \frac{2\alpha }{\beta ^{2}}%
(y-m)-\frac{2\rho }{\beta }e^{y}\right) ^{2}}f(y)dy.
\end{eqnarray*}
We know that $$\lim_{x\rightarrow +\infty }\frac{\frac{2\alpha }{\beta ^{2}}-\frac{2\rho }{%
\beta }e^{x}}{\left( \frac{2\alpha }{\beta ^{2}}(x-m)-\frac{2\rho }{\beta }%
e^{x}\right) ^{2}}=0.$$
Thus  there exists $M>c>0$, such that for $y>M$,
$$\left\vert \frac{\frac{%
2\alpha }{\beta ^{2}}-\frac{2\rho }{\beta }e^{x}}{\left( \frac{2\alpha }{%
\beta ^{2}}(x-m)-\frac{2\rho }{\beta }e^{x}\right) ^{2}}\right\vert <\frac{1%
}{2}.$$
Then $$\int_{c}^{x}f(y)dy\geq 2\left[ \frac{f(y)}{\left( \frac{2\alpha }{\beta ^{2}}%
(y-m)-\frac{2\rho }{\beta }e^{y}\right) }\right] _{c}^{x}$$
Since $
\lim_{y\rightarrow +\infty }\frac{f(y)}{\left( \frac{2\alpha }{\beta ^{2}}%
(y-m)-\frac{2\rho }{\beta }e^{y}\right) }=+\infty$, then $\int_{c}^{\infty }f(y)dy=+\infty $, and  $\widetilde{s}(\infty )=A_{2}\int_{c}^{\infty }f(y)dy=+\infty$,
therefore $\widetilde{\upsilon }(\infty
)=+\infty $   and  $ \widetilde{\upsilon }_{b}(\infty )=+\infty $

\item Case (2): $\rho>0$ \\
We have
\begin{eqnarray*}
\widetilde{s}(\infty )=A_{2}\int_{c}^{\infty }f(y)dy<+\infty.
\end{eqnarray*}
We shall check the finiteness of  $\widetilde{v}(\infty )$, where
\begin{eqnarray*}
\widetilde{\upsilon }(\infty )=\frac{2}{\beta ^{2}}\int_{c}^{\infty } \frac{\int_{y}^{\infty }f(z)dz}{f(y)}
 dy.
\end{eqnarray*}
Since  $\lim_{y\rightarrow \infty }e^{-y}f(y)=0$, by using L'H\^{o}pital's rule, we get
\begin{eqnarray*}
&& \lim_{y\rightarrow +\infty } \frac{\int_{y}^{+\infty }f(z)dz}{e^{-y}f(y)}=
\lim_{y\rightarrow +\infty }\frac{e^{y}}{ 1-\left(
\frac{2\alpha }{\beta ^{2}}(y-m)-\frac{2\rho }{\beta }e^y\right) }
=\frac{ \beta}{2 \rho}
\end{eqnarray*}
Thus as $y\rightarrow +\infty $%
\begin{eqnarray*}
\int_{y}^{+\infty }f(y)dz \sim \frac{ \beta}{2 \rho} e^{-y} f(y)
\end{eqnarray*}
and there exists $M>c>0$, such that for $y>M$,
\begin{eqnarray}\label{maj1}
\int_{y}^{+\infty }f(z)dz \leq \frac{ \beta}{\rho} e^{-y}f(y)
\end{eqnarray}
Taking (\ref{maj1}) into account, we obtain the following
\begin{eqnarray*}
\widetilde{\upsilon }(\infty) &=&\frac{2}{\beta ^{2}}\int_{c}^{\infty}\frac{\int_{y}^{\infty}f(z)dz}{f(y)} dy \\
&=&\frac{2}{\beta ^{2}}\int_{c}^{M}\frac{\int_{y}^{\infty}f(z)dz}{f(y)}dy
+ \frac{2}{\beta ^{2}}\int_{M }^{\infty}\frac{\int_{y}^{\infty}f(z)dz}{f(y)} dy \\
&\leq&\frac{2}{\beta ^{2}}\int_{c }^{M}\frac{\int_{y}^{\infty}f(z)dz}{f(y)} dy + \frac{2}{\beta ^{2}}\int_{M}^{\infty }\frac{\beta}{\rho} e^{-y}dy \\
&\leq&\frac{2}{\beta ^{2}}\int_{c }^{M}\frac{\int_{y}^{\infty}f(z)dz}{f(y)} dy+\frac{2 e^{- M}}{\rho \beta }  <\infty
\end{eqnarray*}
\end{itemize}
The same arguments work for $\widetilde{v}_{b}(\infty)$:
\begin{eqnarray*}
\widetilde{\upsilon }_{b}(\infty )=\frac{2}{\beta ^{2}}\int_{c}^{\infty }  e^{2y} \frac{
\int_{y}^{\infty }f(z)dz}{f(y)}dy
\end{eqnarray*}
Since $\lim_{y\rightarrow \infty }e^{-y}f(y)=0$, applying L'H\^{o}pital's rule, we obtain :
\begin{eqnarray*}
\lim_{y\rightarrow +\infty }\frac{\int_{y}^{+\infty }f(z)dz}{e^{-y}f(y)}=\lim_{y\rightarrow +\infty }\frac{e^{y}}{1-\left(\frac{2\alpha
}{\beta ^{2}}(y-m) -\frac{2\rho }{\beta}e^y\right) }=\frac{\beta}{2 \rho}
\end{eqnarray*}
Thus as $y\rightarrow +\infty $%
\begin{eqnarray*}
\int_{y}^{+\infty }f(z)dz \sim \frac{\beta}{2 \rho} e^{-y} f(y)
\end{eqnarray*}
and there exists $M>c>0$, such that for $y>M$,
$\int_{y}^{+\infty }f(z)dz \geq \frac{ \beta}{4\rho} e^{-y}f(y)$
\begin{eqnarray*}
\widetilde{\upsilon }_{b}(\infty) &=&\frac{2}{\beta ^{2}}\int_{c}^{\infty }  e^{2y} \frac{
\int_{y}^{\infty }f(z)dz}{f(y)}dy\\
&=&\frac{2}{\beta ^{2}}\int_{c}^{M}  e^{2y} \frac{
\int_{y}^{\infty }f(z)dz}{f(y)}dy+ \frac{2}{\beta ^{2}}\int_{M}^{\infty }  e^{2y} \frac{
\int_{y}^{\infty }f(z)dz}{f(y)}dy\\
&>&\frac{2}{\beta ^{2}}\int_{c}^{M}  e^{2y} \frac{
\int_{y}^{\infty }f(z)dz}{f(y)}dy+\frac{2}{\beta ^{2}}\int_{M}^{\infty }\frac{ \beta}{4\rho} e^{y}dy = \infty
\end{eqnarray*}
Then $\widetilde{\upsilon }_{b}(\infty)=\infty$. \\
To summarize,
\begin{eqnarray}
\widetilde{\upsilon }(\infty )=\left\{
\begin{array}{lcr}
+\infty \text{ \ \ \ \ \ \ \ \ \ if }\rho \leq 0 \\
<+\infty \text{ \ \ \ \ \ \ \ if }\rho >0
\end{array}%
\right.
\end{eqnarray}
\begin{eqnarray}
\widetilde{\upsilon }_{b}(\infty )=+\infty \text{ \ \ \ \ \ \ \ }\forall \;  \rho \in \left[ -1,1\right]
\end{eqnarray}
Let us now check the conditions for $\ell$, recall $\ell=0$
\begin{eqnarray*}
\widetilde{s}(0)=-A_{2}\int_{0}^{c}f(y)dy
\end{eqnarray*}
\begin{itemize}
\item Case (1): $\rho \leq 0$ \\
We check the finiteness of $\widetilde{v}(0)$ for this case:
\begin{eqnarray*}
\widetilde{\upsilon }(0)=\frac{2}{\beta ^{2}}\int_{0}^{c}\frac{ \int_{0}^{y}f(z)dz}{f(y)} dy
\end{eqnarray*}
since $\lim_{y\rightarrow 0}y f(y)=0$, and from L'H\^{o}pital's rule we get
\begin{eqnarray*}
\lim_{y\rightarrow 0}\frac{\int_{0}^{y}f(z)dz}{yf(y)}
=\lim_{y\rightarrow 0}\frac{1}{ 1+\left(
\frac{2\alpha }{\beta ^{2}}(y-m)-\frac{2\rho }{\beta }e^{y}\right) y} =1
\end{eqnarray*}
Thus as $y\rightarrow 0$%
\begin{eqnarray*}
\int_{0}^{y}f(z)dz\sim y f(y)
\end{eqnarray*}
Then  there exists $0<\varepsilon <c$ such that $\forall \;  0<y<\varepsilon $, $\int_{0}^{y}f(z) dz \leq 2 y f(y) $
\begin{eqnarray*}
\widetilde{\upsilon }(0) &=&\frac{2}{\beta ^{2}}\int_{0}^{c}\frac{ \int_{0}^{y}f(z)dz}{f(y)} dy\\
&=&\frac{2}{\beta ^{2}}\int_{0}^{\varepsilon}\frac{ \int_{0}^{y}f(z)dz}{f(y)} dy+ \frac{2}{\beta ^{2}}\int_{\varepsilon}^{c}\frac{ \int_{0}^{y}f(z)dz}{f(y)} dy\\
&\leq&\frac{2}{\beta ^{2}}\int_{0}^{\varepsilon }2ydy+ \frac{2}{\beta ^{2}}\int_{\varepsilon}^{c}\frac{ \int_{0}^{y}f(z)dz}{f(y)} dy\\
&\leq&\frac{2\varepsilon ^{2}}{\beta ^{2}}+\frac{2}{\beta ^{2}}\int_{\varepsilon}^{c}\frac{ \int_{0}^{y}f(z)dz}{f(y)} dy <+\infty
\end{eqnarray*}
The same for $\widetilde{v}_{b}(0)$:
\begin{eqnarray*}
\widetilde{\upsilon }_{b}(0)=\frac{2}{\beta ^{2}}\int_{0}^{c}e^{2y} \frac{ \int_{0}^{y}f(z)dz}{f(y)}dy
\end{eqnarray*}
we have $\lim_{y\rightarrow 0}y e^{-2y} f(y)=0$, from L'H\^{o}pital's rule:
\begin{eqnarray*}
\lim_{y\rightarrow 0}\frac{\int_{0}^{y}f(z) dz}{y e^{-2y} f(y)}
=\lim_{y\rightarrow 0}\frac{1}{\left( 1+\left(\frac{2\alpha }{\beta ^{2}}%
(y-m)-\frac{2\rho }{\beta }e^{y}\right)y-2y\right) e^{-2y}} =1
\end{eqnarray*}
Thus as $y\rightarrow 0$%
\begin{eqnarray*}
\int_{0}^{y}f(z) dz\sim y e^{-2y} f(y)
\end{eqnarray*}
Then there exist $0<\varepsilon <c$ \ ; \ $\forall \;  0<y<\varepsilon $, $\int_{0}^{y}f(z) dz \leq 2y e^{-2y} f(y)$

\begin{eqnarray*}
\widetilde{\upsilon }_{b}(0) &=&\frac{2}{\beta ^{2}}\int_{0}^{c}e^{2y} \frac{ \int_{0}^{y}f(z)dz}{f(y)}dy \\
&=&\frac{2}{\beta ^{2}}\int_{0}^{\varepsilon}e^{2y} \frac{ \int_{0}^{y}f(z)dz}{f(y)}dy+\frac{2}{\beta ^{2}}\int_{\varepsilon}^{c}e^{2y} \frac{ \int_{0}^{y}f(z)dz}{f(y)}dy\\
&\leq &\frac{2}{\beta ^{2}}\int_{0}^{\varepsilon }2ydy+\frac{2}{\beta ^{2}}\int_{\varepsilon}^{c}e^{2y} \frac{ \int_{0}^{y}f(z)dz}{f(y)}dy \\
&\leq&\frac{2\varepsilon ^{2}}{\beta ^{2}}+\frac{2}{\beta ^{2}}\int_{\varepsilon}^{c}e^{2y} \frac{ \int_{0}^{y}f(z)dz}{f(y)}dy <+\infty
\end{eqnarray*}
\item Case (2): $\rho>0$ \\
We check the finiteness of $\widetilde{v}(0)$ for this case:
\begin{eqnarray*}
\widetilde{\upsilon }(0)=\frac{2}{\beta ^{2}}\int_{0}^{c} \frac{ \int_{0}^{y}f(z)dz} {f(y)} dy
\end{eqnarray*}
Since $\lim_{y\rightarrow 0}y f(y)=0$, we apply L'H\^{o}pital's rule:%
\begin{eqnarray*}
\lim_{y\rightarrow 0}\frac{\int_{0}^{y}f(z)dz}{y f(y)}=\lim_{y\rightarrow 0}\frac{1}{ 1+\left(
\frac{2\alpha }{\beta ^{2}}(y-m)-\frac{2\rho }{\beta }e^{y}\right) y}=1.
\end{eqnarray*}
Thus as $y\rightarrow 0$
\begin{eqnarray*}
\int_{0}^{y}f(z)dz\sim y f(y)
\end{eqnarray*}
Then there exists $0<\varepsilon <c$ such that $\forall \;  0<y<\varepsilon $, $\int_{0}^{y}f(y)dz\leq 2yf(y)$, therefore
\begin{eqnarray*}
\widetilde{\upsilon }(0) &=&\frac{2}{\beta ^{2}}\int_{0}^{c} \frac{ \int_{0}^{y}f(z)dz} {f(y)} dy \\
&=&\frac{2}{\beta ^{2}}\int_{0}^{\varepsilon} \frac{ \int_{0}^{y}f(z)dz} {f(y)} dy +\frac{2}{\beta ^{2}}\int_{\varepsilon}^{c} \frac{ \int_{0}^{y}f(z)dz} {f(y)} dy\\
&\leq &\frac{2}{\beta ^{2}}\int_{0}^{\varepsilon }2ydy+\frac{2}{\beta ^{2}}\int_{\varepsilon}^{c} \frac{ \int_{0}^{y}f(z)dz} {f(y)} dy \\
&\leq&\frac{2\varepsilon ^{2}}{\beta ^{2}}+\frac{2}{\beta ^{2}}\int_{\varepsilon}^{c} \frac{ \int_{0}^{y}f(z)dz} {f(y)} dy<+\infty
\end{eqnarray*}
The same for $\widetilde{v}_{b}(0)$:
\begin{eqnarray*}
\widetilde{\upsilon }_{b}(0)=\frac{2}{\beta ^{2}}\int_{0}^{c}e^{2y}\frac{ \int_{0}^{y}f(z)dz} {f(y)} dy
\end{eqnarray*}
we have $\lim_{y\rightarrow 0}y e^{2y} f(y)=0 $, from L'H\^{o}pital's rule:%
\begin{eqnarray*}
\lim_{y\rightarrow 0}\frac{\int_{0}^{y}f(y)dz}{ye^{-2y}f(y)}=\lim_{y\rightarrow 0}\frac{1}{\left( 1+\frac{2\alpha }{\beta ^{2}}%
(y-m)y-2y \right)e^{2y} }=1
\end{eqnarray*}
thus as $y\rightarrow 0$%
\begin{eqnarray*}
\int_{0}^{y}f(z)dz\sim ye^{-2y}f(y),
\end{eqnarray*}
then there exists $0<\varepsilon <c$ such that $\forall \;  0<y<\varepsilon $, $\int_{0}^{y}f(z)dz \leq 2ye^{-2y}f(y)$, hence
\begin{eqnarray*}
\widetilde{\upsilon }_{b}(0) &=&\frac{2}{\beta ^{2}}\int_{0}^{c}e^{2y}\frac{ \int_{0}^{y}f(z)dz} {f(y)} dy \\
&=&\frac{2}{\beta ^{2}}\int_{0}^{\varepsilon}e^{2y}\frac{ \int_{0}^{y}f(z)dz} {f(y)} dy +
\frac{2}{\beta ^{2}}\int_{\varepsilon}^{c}e^{2y}\frac{ \int_{0}^{y}f(z)dz} {f(y)} dy \\
&\leq&\frac{2}{\beta ^{2}}\int_{0}^{\varepsilon }2ydy+\frac{2}{\beta ^{2}}\int_{\varepsilon}^{c}e^{2y}\frac{ \int_{0}^{y}f(z)dz} {f(y)} dy \\
&\leq&\frac{2\varepsilon ^{2}}{\beta ^{2}}+\frac{2}{\beta ^{2}}\int_{\varepsilon}^{c}e^{2y}\frac{ \int_{0}^{y}f(z)dz} {f(y)} dy \\
&<&\infty
\end{eqnarray*}
\end{itemize}
To summarize,%
\begin{eqnarray}
\widetilde{\upsilon }(0 )<+\infty \text{ \ \ \ \ \ \ \ }\forall \;  \rho
\in \left[ -1,1\right]
\end{eqnarray}
\begin{eqnarray}
\widetilde{\upsilon }_{b}(0)<+\infty \text{ \ \ \ \ \ \ \ }\forall \;  \rho
\in \left[ -1,1\right]
\end{eqnarray}


\appendix{Proof of Theorem \ref{T10}}\label{Asubsec:2}
For ease of notations we denote by $g(x)=e^{\frac{\alpha }{\beta ^{2}}(x-m)^{2}}$.
Under $\mathbb{P}$, we have a scale function:
\begin{eqnarray*}
s(x)=A_{1}\int_{c}^{x}g(y)dy
\end{eqnarray*}
To check the conditions for $r$, recall $r=\infty$
\begin{eqnarray*}
s(\infty )=A_{1}\int_{c}^{\infty }g(y)dy>A_{1}\int_{c}^{\infty }e^{\frac{\alpha }{\beta ^{2}}%
(y-m)}dy=A_{1}\left[ \frac{\beta ^{2}}{\alpha }e^{\frac{\alpha }{\beta ^{2}}%
(y-m)}\right] _{c}^{\infty }=+\infty
\end{eqnarray*}
Since $s(\infty )=+\infty$, then \ $\upsilon (\infty )=+\infty $ \ and $\upsilon_{b}(\infty )=+\infty $ \\
To check similar conditions for $\ell$, recall $\ell=0$
\begin{eqnarray*}
 s(0)=-A_{1}\int_{0}^{c}g(y)dy>-\infty
\end{eqnarray*}
We check the finiteness of $v(0)$ for this case:
\begin{eqnarray*}
\upsilon (0)=\frac{2}{\beta ^{2}}\int_{0}^{c}\frac{ \int_{0}^{y}g(z)dz}{g(y)} dy
\end{eqnarray*}
One has $\lim_{y\rightarrow 0}yg(y)=0$, from L'H\^{o}pital's rule:%
\begin{eqnarray*}
\lim_{y\rightarrow 0}\frac{\int_{0}^{y}g(z)dz}{y g(y)}=\lim_{y\rightarrow 0}%
\frac{1}{ 1+\frac{2\alpha }{%
\beta ^{2}}(y-m)y}=1
\end{eqnarray*}
hence, as $y\rightarrow 0$%
\begin{eqnarray*}
\int_{0}^{y}g(z)dz \sim y g(y)
\end{eqnarray*}
hence there exists $0<\varepsilon <c$ such that $\forall \;  0<y<\varepsilon $, $\int_{0}^{y}g(z)dz \leq 2y g(y)$
\begin{eqnarray*}
\upsilon (0) &=&\frac{2}{\beta ^{2}}\int_{0}^{c}\frac{ \int_{0}^{y}g(z)dz}{g(y)} dy\\
&=&\frac{2}{\beta ^{2}}\int_{0}^{\varepsilon}\frac{ \int_{0}^{y}g(z)dz}{g(y)} dy+\frac{2}{\beta ^{2}}\int_{\varepsilon}^{c}\frac{ \int_{0}^{y}g(z)dz}{g(y)} dy\\
&\leq&\frac{2}{\beta ^{2}}\int_{0}^{\varepsilon }2ydy+\frac{2}{\beta ^{2}}\int_{\varepsilon}^{c}\frac{ \int_{0}^{y}g(z)dz}{g(y)} dy \\
&\leq& \frac{2\varepsilon ^{2}}{\beta ^{2}}+\frac{2}{\beta ^{2}}\int_{\varepsilon}^{c}\frac{ \int_{0}^{y}g(z)dz}{g(y)} dy <+\infty
\end{eqnarray*}
The same for $v_{b}(0)$:
\begin{eqnarray*}
\upsilon _{b}(0)=\frac{2}{\beta ^{2}}\int_{0}^{c}e^{2y} \frac{ \int_{0}^{y}g(z)dz}{g(y)}dy
\end{eqnarray*}
we have $\lim_{y\rightarrow 0}y e^{-2y} g(y)=0$, by using L'H\^{o}pital's rule, we get:
\begin{eqnarray*}
\lim_{y\rightarrow 0}\frac{\int_{0}^{y}g(z)dz}{ye^{-2y}g(y)}%
=\lim_{y\rightarrow 0}\frac{1}{\left(
1+\frac{2\alpha }{\beta ^{2}}(y-m)y-2y\right)e^{-2y} }=1
\end{eqnarray*}
Thus as $y\rightarrow 0$%
\begin{eqnarray*}
\int_{0}^{y}g(z)dz\sim ye^{-2y} g(y).
\end{eqnarray*}
Therefore one can choose $0<\varepsilon <c$ such that $\forall \;  0<y<\varepsilon $, $\int_{0}^{y}g(z)dz\leq2y e^{-2y} g(y)
$
\begin{eqnarray*}
\upsilon _{b}(0) &=&\frac{2}{\beta ^{2}}\int_{0}^{c}e^{2y} \frac{ \int_{0}^{y}g(z)dz}{g(y)}dy\\
&=&\frac{2}{\beta ^{2}}\int_{0}^{\varepsilon}e^{2y} \frac{ \int_{0}^{y}g(z)dz}{g(y)}dy + \frac{2}{\beta ^{2}}\int_{\varepsilon}^{c}e^{2y} \frac{ \int_{0}^{y}g(z)dz}{g(y)}dy\\
&\leq&\frac{2}{\beta ^{2}}\int_{0}^{\varepsilon }2ydy+\frac{2}{\beta ^{2}}\int_{\varepsilon}^{c}e^{2y} \frac{ \int_{0}^{y}g(z)dz}{g(y)}dy \\
&\leq&\frac{2\varepsilon ^{2}}{\beta ^{2}}+\frac{2}{\beta ^{2}}\int_{\varepsilon}^{c}e^{2y} \frac{ \int_{0}^{y}g(z)dz}{g(y)}dy\\
&<& +\infty
\end{eqnarray*}


\begin{thebibliography}{99}
\bibitem{and}
{L. Andersen \and V. Piterbarg}, {\it Moment explosions in stochastic volatility models,
Finance and Stochastics,} (2007), 11, pp. 29--50.

\bibitem{bay}
{E. Bayraktar, C. Kardaras \and H. Xing},  {\it Valuation equations for stochastic volatility models,} SIAM Journal on Financial Mathematics, (2012), 3, pp. 351-373.

\bibitem{Bernard}
{C. Bernard, Z. Cui \and D.L. McLeish},  {\it On the martingale property in stochastic volatility models based on time--homogeneous diffusions,} (2014), arXiv:1310.0092v1.

\bibitem{BS}
{F. Black \and M. Scholes},  {\it The Pricing of Options and Corporate Liabilities,} The Journal of Political Economy 81,(1973), pp. 637--654.

\bibitem{CFMN}
{E. Cisana, L. Fermi, G. Montagna \and O. Nicrosini}, {\it A Comparative Study of Stochastic Volatility Models}, (2007), arXiv:0709.0810v1. 

\bibitem{EPM}
{Z. Eisler, J. Perell\'{o} \and J. Masoliver}, {\it Volatility, a hidden Markov process in financial time series,} Physical Review E 76, (2007), 056105.

\bibitem{ES1991}
{H.-J. Engelbert \and W. Schmidt}, {\it Strong Markov continuous local martingales and solutions of one--dimensional stochastic differential equations}, Springer, Berlin, (1991).

\bibitem{ES1985}
{H.-J. Engelbert \and W. Schmidt}, {\it On one--dimensional stochastic differential equations with generalized drift}, Springer, Berlin, (1985).

\bibitem{HESTON}
{S. Heston}, {\it A closed--form solution for options with stochastic volatiliy with applications to bond and currency options,}  Review of Financial Studies 6, (1993), pp. 327--343.

\bibitem{HW}
{J. C. Hull \and A. White}, {\it The pricing of options on assets with stochastic volatilities,}  The Journal of Finance 42, (1987), pp. 281--300.

\bibitem{j}
{B. Jourdain}, {\it Loss of martingality in asset price models with lognormal stochastic volatility}, 2004.

\bibitem{LM}
{P. Lions \and M. Musiela, } {\it Correlations and bounds for stochastic volatility models},  Annals of Institute of Henri Poincare  24, (2007), pp. 1--16.

\bibitem{MP3}
{J. Masoliver \and J. Perell\'{o}}, {\it Extreme times for volatility process,}  Physical Review E. 75, (2007), 046110.

\bibitem{MT}
{A. Melino \and S. Turnbull}, {\it Pricing Foreign Currency Options With Stochastic Volatility,} Journal of Econometrics 45, (1990), pp. 7--39.

\bibitem{MU}
{A. Mijatovic \and M. Urusov}, {\it On the martingale property of certain local martingales,}  Probability Theory and Related Fields 152, (2012), pp. 1--30.

\bibitem{QZ}
{T. Qiu, ,B. Zheng, F. Ren, \and S. Trimper}, {\it  Return--volatility correlation in financial dynamics,}  Physical Review E 73, (2006), 065103.

\bibitem{R}
{J. Ruf}, {\it A new proof for the conditions of Novikov and Kazamaki,} Stochastic Processes and their Applications 23, (2013), pp. 404--421.

\bibitem{Scott}
{L. Scott}, {\it Option pricing when the variance changes randomly: theory, estimation and an
application,}  Journal of Financial and Quantitative Analysis  22, (1987), pp. 419--438.

\bibitem{Sin}
{C. Sin}, {\it Complications with stochastic volatility models}, Advances in Applied Probability, (1998), 
30(1), pp. 256--268.

\bibitem{SS}
{E. Stein \and J. Stein}, {\it Stock Price Distributions with Stochastic Volatility: An Analytic Approach,}  Review of Financial Studies 4, (1991), pp. 727--752.

\bibitem{Wiggins}
{J. Wiggins}, {\it Option Values Under Stochastic Volatility: Theory and Empirical Estimates,} Journal of Financial Economics 19, (1987), pp. 351--372.

\bibitem{yor}
{C. Wu \and M. Yor}, {\it Linear transformations of two independent Brownian motions and orthogonal decompositions of Brownian filtrations,}  Publ. Mat. 46, (2002), pp. 237--256
\end{thebibliography}
\end{document}